\documentclass[english]{amsart}

\usepackage{amsmath,amssymb,enumerate}

\usepackage[T1]{fontenc}
\usepackage[all]{xy}

\usepackage{babel}
\usepackage{amstext}
\usepackage{amsmath}
\usepackage{amsfonts}
\usepackage{latexsym}
\usepackage{ifthen}

\usepackage{xypic}
\xyoption{all}
\newcommand{\im}{{\rm im}}

\newtheorem{lemma1}{}[section]

\newenvironment{lemma}{\begin{lemma1}{\bf Lemma.}}{\end{lemma1}}
\newenvironment{example}{\begin{lemma1}{\bf Example.}\rm}{\end{lemma1}}

\newenvironment{theorem}{\begin{lemma1}{\bf Theorem.}}{\end{lemma1}}
\newenvironment{proposition}{\begin{lemma1}{\bf Proposition.}}{\end{lemma1}}
\newenvironment{corollary}{\begin{lemma1}{\bf Corollary.}}{\end{lemma1}}

\newenvironment{definition}{\begin{lemma1}{\bf Definition.}}{\end{lemma1}}

\newenvironment{conjecture}{\begin {lemma1}{\bf Conjecture.}}{\end{lemma1}}

\newenvironment{remark*}{{\bf Remark.}}{}
\newenvironment{example*}{{\bf Example.}}{}

\newcommand{\Z}{\ensuremath{\mathbb{Z}}}
\newcommand{\C}{\ensuremath{\mathbb{C}}}

\newcommand{\merom}[3]{\ensuremath{#1:#2 \dashrightarrow #3}}

\newcommand{\holom}[3]{\ensuremath{#1:#2  \rightarrow #3}}

\makeatletter
\ifnum\@ptsize=0  \fi \ifnum\@ptsize=2  \fi 

\newcommand{\Chow}[1]{\ensuremath{{\rm Chow}(#1)}}

\newcommand{\upX}{\ensuremath{\tilde{X}}}

\newcommand{\upY}{\ensuremath{\tilde{Y}}}

\newcommand{\barX}{\ensuremath{\overline{X}}}

\newcommand{\barY}{\ensuremath{\overline{Y}}}

\newtheorem{say}[lemma1]{}

\renewcommand{\c}[0]{{\mathbb C}}  

\renewcommand{\o}[0]{{\mathcal O}} 
\newcommand{\z}[0]{{\mathbb Z}}

\newcommand{\p}[0]{{\mathbb P}}

\newcommand{\q}[0]{{\mathbb Q}}
\newcommand{\map}[0]{\dasharrow}
\newcommand{\qtq}[1]{\quad\mbox{#1}\quad}

\newcommand{\aut}[0]{\operatorname{Aut}}

\newcommand{\sing}[0]{\operatorname{Sing}}    
\newcommand{\ex}[0]{\operatorname{Ex}}

\newcommand{\stab}[0]{\operatorname{stab}}

\newcommand{\mor}[0]{\operatorname{FinMor}} 
\newcommand{\univ}[0]{\operatorname{Univ}}

\newcommand{\loc}[0]{\operatorname{Locus}} 
\newcommand{\alb}[0]{\operatorname{alb}} 
\newcommand{\Alb}[0]{\operatorname{Alb}}

\def\into{\DOTSB\lhook\joinrel\to}

\title{Algebraic varieties with quasi-projective universal cover}
\date{February 14, 2011}

\subjclass[2000]{32Q30, 14E30, 14J30}
\keywords{universal cover, MMP}

\author{Beno\^it Claudon}
\author{Andreas H\"oring}
\author{J\'anos Koll\'ar}

\address{Beno\^it Claudon, Institut \'Elie Cartan Nancy, Universit\'e Henri Poincar\'e Nancy 1, B.P. 70239, 54506 Vandoeuvre-l\`es-Nancy Cedex, France}
\email{Benoit.Claudon@iecn.u-nancy.fr}

\address{Andreas H\"oring, Universit\'e Pierre et Marie Curie and Albert-Ludwig Universit\"at Freiburg} 
\curraddr{Mathematisches Institut, Albert-Ludwigs-Universit\"at
  Freiburg, Eckerstra{\ss}e 1, 79104 Freiburg im Breisgau, Germany}
\email{hoering@math.jussieu.fr}

\address{J\'anos Koll\'ar,  Department of Mathematics, Fine Hall, Washington Road,
Princeton University,  Princeton, NJ 08544-1000,  USA}
\email{kollar@math.princeton.edu}

\begin{document}

\begin{abstract} 
We prove that the universal cover of a normal, projective variety $X$ is
quasi-projective iff a finite, \'etale cover of $X$ is a 
fiber bundle over an Abelian variety with simply connected fiber. 
\end{abstract}

\maketitle


\vspace{-1ex}

\section{Introduction}

In his book \cite[Sec.IX.4]{Sha74}, Shafarevich emphasizes the
need to understand universal covers of smooth, projective varieties.
Although his  conjectures may not hold in general \cite{BK98},
they are true for groups with faithful linear representations
\cite{EKPR}. Applications of the general ideas behind the Shafarevich
conjectures are discussed in \cite{kol-book}.
These methods are especially powerful if the universal cover
is easy to describe, as it happens for Abelian varieties,
whose universal cover is $\c^n$.
This suggests that one should study  projective varieties
whose  universal cover is quasi-projective.

There are two significant results in this direction.

$\bullet$ Nakayama shows \cite[Thm.1.4]{Nak99} that 
the  universal cover of a smooth, 
projective variety $X$
is  quasi affine if and only if $X$ has a finite \'etale cover that is 
 an  Abelian variety.

$\bullet$ It is a consequence of the Beauville--Bogomolov decomposition 
theorem \cite{Beau83} that if $X$ is a Calabi--Yau variety, then
its universal cover $\tilde X$ is biholomorphic to $Y\times \c^m$ where
$Y$ is a compact, simply connected, Calabi--Yau variety.
In particular, $\pi_1(X)$ is almost abelian and $\tilde X$ is
  biholomorphic to a  quasi-projective variety.

In this paper we prove the following theorem which can be viewed as a common
generalization of these results.
While we give a complete answer, the proof assumes the validity
of  the abundance conjecture (\ref{abund.conj}).
This assumption is already present   in  \cite{Nak99} explicitly
and  in \cite{Beau83} implicitly.

\begin{theorem}  \label{theoremquasiprojective}
Assume that the abundance conjecture (\ref{abund.conj}) holds.
Then, for any  normal, projective variety $X$ the following are equivalent.
\begin{enumerate}
\item The universal cover of $X$ is  biholomorphic to a  
quasi-projective variety.
\item There is a finite, \'etale, Galois  cover $X'\to X$ that is a 
fiber bundle over an Abelian variety with
simply connected fiber. 
\item The universal cover $\upX$ is biholomorphic to a 
product $\C^m \times F$ where $m\geq 0$
and $F$ is a projective, simply connected variety.
\end{enumerate}
\end{theorem}

Note, however, that 
in general there is no finite, \'etale, Galois  cover $X'\to X$ 
that is a product of  an Abelian variety with $F$.

We want to emphasize that in (\ref{theoremquasiprojective}.1)
we do not assume that the deck transformations are algebraic
automorphisms of the  universal cover and in fact this is not true in general
(\ref{nonalg.action.exmp}).

It is clear that (\ref{theoremquasiprojective}.3)  $\Rightarrow$
(\ref{theoremquasiprojective}.1) and (\ref{theoremquasiprojective}.2) implies
(\ref{theoremquasiprojective}.3)  since every
fiber bundle over $\C^m$ with compact analytic fiber is trivial
(\ref{theoremisotrivial}).

The proof of  (\ref{theoremquasiprojective}.1) $\Rightarrow$
(\ref{theoremquasiprojective}.2) comes in two independent steps,
both of which are more general than needed for (\ref{theoremquasiprojective}).
First we  show in (\ref{fg.ab.prop}) that
the fundamental group of $X$ is almost abelian, that is, it
 contains an abelian subgroup of finite index.
Here we
 use the  abundance conjecture to rule out some possible counter examples.

Once we know that the fundamental group of $X$ is almost abelian,
by passing to a finite cover we may assume that
it is in fact abelian. Then we prove directly in (\ref{alb.general.thm})
 that the 
 Albanese morphism is a fiber bundle.
 This part does not rely on  any conjectural assumptions.

\begin{proposition}\label{fg.ab.prop}
 Let $X$ have the smallest dimension among all  normal, projective varieties
that have  an infinite,  quasi-projective, Galois cover $\tilde X\to X$
whose Galois group is not  almost abelian.

Then $X$ is smooth and  its canonical class $K_X$ is nef but not semi-ample.
(That is, $(K_X\cdot C)\geq 0$ for every algebraic curve $C\subset X$
but $\o_X(mK_X)$ is not generated by global sections for any $m>0$.)
\end{proposition}

The conclusion would contradict the following, so called
{\it abundance conjecture} \cite[Sec.2]{reid-minmod}.
 Thus if (\ref{abund.conj}) holds then
$X$ as in (\ref{fg.ab.prop}) can not exist. 
Thus if a   normal, projective variety $Y$  has an 
 infinite,  quasi-projective, Galois cover $\tilde Y\to Y$
then the Galois group is  almost abelian.

\begin{conjecture}\label{abund.conj}
Let $X$ be a smooth projective variety such that  $K_X$ is nef.
Then $K_X$ is  semi-ample.
\end{conjecture}

Note that the abundance conjecture is frequently stated for
varieties with log canonical singularities, even for  log canonical
pairs, but we need only the smooth case.
 The conjecture is known to hold if
$\dim X\leq 3$; see \cite{f&a} for a detailed treatment.

Next we study the  quasi-projectivity of Abelian covers.

\begin{theorem}\label{alb.general.thm} Let $X$ be a normal, projective variety
and $\alpha:X\to A$ a morphism to an Abelian variety.
Let $\pi:\tilde A\to A$ be an \'etale Galois 
cover with group $\Gamma$ such that $\tilde A$ has no compact analytic subvarieties.
 By pull-back we obtain
$\tilde\alpha:\tilde X\to \tilde A$.

If $\tilde X$ is quasi-projective then $\alpha:X\to A$ is a
locally trivial fiber bundle. 
\end{theorem}


In  light of the previous statements, we can strengthen Theorem \ref{theoremquasiprojective} as follows:

\begin{corollary}\label{better.main.thm} 
Assume that the abundance conjecture (\ref{abund.conj}) holds.
Let $X$ be a normal, projective variety and $\tilde X\to X$
an infinite \'etale Galois cover such that $\tilde X$ is  quasi-projective. 
Then there exist
\begin{enumerate}
\item  a finite, \'etale, Galois  cover $X'\to X$, 
\item a morphism to an Abelian variety $\alpha:X'\to A$ which is a 
locally trivial fiber bundle  and
\item an \'etale cover $\pi: \tilde A\to A$ such that $\tilde A$ has no 
compact analytic subvarieties
\end{enumerate}
such that  $\tilde X$ pulls-back from $ \tilde A\to A$, that is
$\tilde X\simeq X'\times_A \tilde A.$
\end{corollary}

We do not know if the converse of (\ref{better.main.thm})
holds or not. Every \'etale cover  $\tilde A\to A$ of an Abelian variety is 
 quasi-projective \cite{CC91}. Note, however, that it can happen that
 $\tilde A$ has no 
compact analytic subvarieties yet it is not Stein \cite{AK01}.
 Even if $X\to A$ is a $\p^1$-bundle,
we do not know if $X\times_A \tilde A$  is  quasi-projective or not.

\begin{say}[The non-algebraic case]

{\rm More generally, it is interesting to
study compact complex manifolds $M$ whose universal cover
$\tilde M$ is a Zariski open submanifold of a 
compact complex manifold $\bar M$.  Besides the algebraic cases,
such examples are given by Hopf manifolds, compact nilmanifolds or more generally any quotient $G\backslash\Gamma$ where $G$ is a (simply connected) non-commutative linear algebraic group and $\Gamma$ a cocompact lattice (see for instance \cite[\S 3.4 and 3.9]{Ak95}, \cite{Win98}). Their classification seems rather difficult.

The problem becomes much more tractable if one assumes that
$M$ (and possibly also $\bar M$) are K\"ahler.
We expect that in this case (\ref{theoremquasiprojective})
should hold, but several steps of the proof need to be changed.
We plan to discuss these in a subsequent paper.}
\end{say}

\begin{say}[The quasi-projective case]
{\rm Our methods rely  on the study of
compact subvarieties of $\tilde X$, but it is possible that
similar results hold if $X$ is quasi-projective.
Very little seems to be known. For instance, we do not know
which  quasi-projective varieties $X$ have $\C^n$ as their universal cover.
The obvious guess is that every such $X$ has a 
finite, \'etale, Galois  cover $X'\to X$
such that $X'\cong \C^n/\Z^m$ where $m\leq 2n$ and
$\Z^m$ acts on $\C^n$ by translations.

The strongest result would be the following analog of 
(\ref{better.main.thm}).

{\it Question.} Let $X$ be a normal, quasi-projective variety and $\tilde X\to X$
an infinite \'etale Galois cover such that $\tilde X$ is  quasi-projective. 
Does there exist 
\begin{enumerate}
\item a finite, \'etale, Galois  cover $X'\to X$,
\item a morphism to a quasi-projective abelian group $\alpha:X'\to A$ that
 is a locally trivial fiber bundle and
\item an \'etale cover $\tilde A\to A$ such that $\tilde A$ has no compact analytic subvarieties
\end{enumerate}
such that $\tilde X\simeq X'\times_A \tilde A$ ?
}
\end{say} 

{\bf Acknowledgements.} The authors want to thank D.~Greb, T.~Peternell and C.~Voisin for  useful comments and references. 
B.C. and A.H. were partially supported by the A.N.R. project ``CLASS''.
Partial financial support for J.K.  was provided by  the NSF under grant number 
DMS-0758275.

\section{Algebraic subvarieties of universal covers}

Let $X$ be a projective variety and $\pi:\tilde X\to X$
an infinite Galois cover with group $\Gamma$
such that $\tilde X$ is  biholomorphic to a  quasi-projective variety.
There is no reason to assume that such a  quasi-projective variety
is unique. In what follows, we fix one such  quasi-projective variety
and say that  $\tilde X$ is
quasi-projective.
If $Z\subset X$ is a closed subvariety, then
its preimage $\tilde Z:=\pi^{-1}(Z)\subset \tilde X$ is 
a closed, analytic subspace of  $\tilde X$,
but it is rarely quasi-projective.

For instance, let $A$ be an Abelian variety
with universal cover $\pi:\c^n\to A$.
If $Z\subsetneq A$ is a closed subvariety, then
its preimage $\tilde Z\subset \c^n$ is never quasi-projective 
by (\ref{alb.general.thm}).
Similarly, let $E$ be a rank 2 vector bundle on $A$
that is not an extension of 2 line bundles.
Set $X=\p_A(E)$ with universal cover $\pi:\tilde X\to X$.
One can see that  if $Z\subsetneq X$ is a closed subvariety, then
its preimage $\tilde Z\subset \c^n$ is never quasi-projective.

We aim to exploit this scarcity of  $\Gamma$-invariant
subvarieties as follows.
If $\tilde X$ has no positive dimensional compact subvarieties
then we are done by  \cite{Nak99} 
(though this is not how our proof actually goes). 
Thus let $F\subset \tilde X$ be a  positive dimensional 
compact subvariety. 
Let $\loc(F, \tilde X) \subset\tilde X$ denote the union of the images
of all finite morphisms $ F\to \tilde X$.
(We are mainly interested in embeddings $ F\into \tilde X$,
but allowing finite maps $ F\to \tilde X$ works better under
finite \'etale covers. We restrict to finite maps mostly  
to avoid the constant maps $ F\to \tilde X$.)
It is clear that $\loc(F, \tilde X)$ is   $\Gamma$-invariant.
Unfortunately, in general we can only prove
that  $\loc(F, \tilde X)$  is  a {\em countable} union of  (locally closed)
algebraic subvarieties of $\tilde X$.
There are, however, 2 special cases where we show that
$\loc(F, \tilde X)$  is  a  (possibly reducible and locally closed)
algebraic subvariety of $\tilde X$.
If $\dim \loc(F, \tilde X)<\dim  \tilde X$, then we use induction to
describe $\loc(F, \tilde X)$ and arrive at a contradiction.
 If $\dim \loc(F, \tilde X) =\dim  \tilde X$, then
we obtain a strong structural description of $ \tilde X$.

\begin{definition} \label{definitionmor}
Let $U$ and $V$ be normal, quasi-projective varieties, and  $U\to V$  a 
flat, projective morphism with a relatively ample divisor $H_V$.
Let $Y$ be a normal, quasi-projective variety and $L$ the restriction of 
an ample line bundle on
some completion $Y \subset \barY$ to $Y$.

We denote  by $\mor(U/V, Y, H_V, L, d)$ the moduli space of finite  morphisms
$\phi:U_v\to Y$ of degree $d$, that is,  if $U_v$ is a fiber
and $\Gamma_\phi \subset U \times Y$ is the graph of $\phi$, then
\[
\bigl(p_U^* H_V+ p_Y^* L \bigr)^{\dim U/V} \cdot \Gamma_\phi=d.
\]
\end{definition}
Note that our ``degree'' is not the degree of the image of $\phi$,
rather the degree of the graph of $\phi$.  
Since the (relative) cycle spaces $\Chow{U/V}$ and $\Chow{\barY}$ are 
projective (over the base $V$)
and the property of being a graph of a morphism is open in the 
Zariski topology, we see that  
$\mor(U/V, Y, H_V, L,  d)$ 
is a quasi-projective subvariety of $\Chow{U\times \bar Y/V}$.
(This would fail if we considered only the degree of the image of $\phi$.)

In order to simplify the notation, we will abbreviate 
$\mor(U/V, Y, H_V, L,  d)$ by
$\mor(U/V, Y, d)$. We have universal families and morphisms
$$
\univ(U/V, Y, d)\to \mor(U/V, Y, d)\qtq{and} \Phi_d: \univ(U/V, Y, d)\to Y.
$$
Set $\mor(U/V, Y):=\bigcup_d\mor(U/V, Y, d)$ with  universal family
$$
\mor(U/V, Y)\leftarrow \univ(U/V, Y)\stackrel{\Phi}{\to} Y.
$$
Note that $\mor(U/V, Y)$ and $ \univ(U/V, Y)$ are, in general,
countable unions of  quasi-projective varieties.

The union of all the images of fibers of $U\to V$
by degree $d$ maps
$$
\loc(U/V, Y, d):=\Phi_d\bigl(\univ(U/V, Y, d)\bigr)\subset Y
$$
 is a constructible algebraic
subset of $Y$ and
$$
\loc(U/V, Y):=\Phi\bigl(\univ(U/V, Y)\bigr)\subset Y
$$ is, in general, 
a countable union of constructible algebraic
subsets.

Our main interest is in the case $Y=\tilde X$ where
$\tilde X\to X$ is Galois with group $\Gamma$.
If we can take $L$ to be $\Gamma$-equivariant
then each $\loc(U/V, \tilde X, d)$ is  $\Gamma$-invariant.
We see, however, no a priori reason why this should be possible.
First, since the $\Gamma$-action is holomorphic but in general
not algebraic, we do not even know that pulling back by
$\gamma\in \Gamma$ maps an algebraic coherent sheaf on $\tilde X$ to
an  algebraic coherent sheaf.
Second, even if we know that  the $\Gamma$-action is  algebraic,
there need not be any  $\Gamma$-equivariant ample line bundles.

\begin{say} {\bf The main construction.} \label{main.1.constr}
{\rm
Let $U$ and $V$ be normal, quasi-projective varieties, and  $U\to V$  a 
flat, projective morphism with a relatively ample divisor $H_V$.
Let $X$ be a projective variety  and 
$\pi:\tilde X\to X$  an infinite Galois cover
with group $\Gamma$ such that $\upX$ is quasi-projective. 
Let furthermore $L$ be a line bundle that is the restriction of an ample line
 bundle on some 
completion $\upX \subset \barX$ to $\upX$. 

Consider $\mor_{\ell}(U/V, X) \subset \mor(U/V, X)$, parametrizing 
those  morphisms
$\phi:U_v\to X$ that can be lifted  to $\tilde\phi:U_v\to \tilde X$.
Note that
$\Gamma$ acts freely on $ \mor(U/V, \tilde X)$ and
we have a natural holomorphic map
$$
\pi_M:\mor(U/V, \tilde X)\to\mor(U/V, \tilde X)/\Gamma= \mor_{\ell}(U/V, X).
$$
Let $W \subset \mor_{\ell}(U/V, X)$ be an irreducible component
and $\tilde W:=\pi_M^{-1}(W)\subset \mor(U/V, \tilde X)$ its preimage.
Every irreducible component of  $\tilde W$ 
is quasi-projective, but usually there are  infinitely many 
and $\Gamma$ permutes them.
Thus we do not get any new algebraic variety with
$\Gamma$-action.

There are, however, two important cases when such a 
 $\tilde W$ has finitely many irreducible components, hence
is itself quasi-projective. We discuss these in 
(\ref{min.constr.lem}) and (\ref{main.constr.bir}). }
\end{say}

\begin{lemma}\label{min.constr.lem}
Let $W \subset \mor_{\ell}(U/V, X)$ be an irreducible component
and   $W\leftarrow \univ_W\to X$  the corresponding universal family.
Assume that $\univ_W\to X$  is dominant.

Then  $\tilde W:=\pi_M^{-1}(W)\subset \mor(U/V, \tilde X)$
is quasi-projective. Moreover, if $V'\subset V$ is an algebraic subvariety
and $U'\to V'$ the corresponding family then
$\tilde W\cap  \mor(U'/V', \tilde X)$ is also   quasi-projective.
\end{lemma}

\begin{proof}
We denote by $\univ_{\tilde W}$ the fiber product 
$\univ_W \times_X  \tilde X$ and
by $\tilde W$ the Stein factorisation of the map 
$ \univ_{\tilde W} \rightarrow \univ_W \rightarrow W$,
so we get a commutative diagram:
\[
\xymatrix{
\tilde W\ar[d] &   \univ_{\tilde W}  \ar[l] \ar[d]   \ar[r]  & \upX \ar[d]_\pi 
\\
  W   &  \univ_W \ar[l] \ar[r]  & X
}
\]
By construction each fiber of $ \univ_{\tilde W} \rightarrow \tilde W$
is also a fiber of  $U\to V$.
Since $\univ_W\to X$ is dominant,  the image of
 $\pi_1(\univ_W)\to \pi_1(X)$
has finite index in $\pi_1(X)$.
 Therefore $ \univ_{\tilde W}$ has only finitely many irreducible 
components, so $\tilde W$ has finitely many
irreducible components. Thus   $\tilde W$  is quasi-projective.

For each fixed $d$, the space of morphisms
 $\mor(U'/V', \tilde X, d)$ is an algebraic subset of 
$\mor(U/V, \tilde X, d)$. 
Since  $\tilde W$ has finitely many
irreducible components, 
 $\tilde W\cap  \mor(U'/V', \tilde X)$
is a closed algebraic subset of a quasi-projective variety,
 hence itself quasi-projective.
\end{proof}

The following consequence will be used repeatedly.

\begin{lemma}\label{main.constr}
Let $X$ be a projective variety and $\pi:\tilde X\to X$ 
 an infinite Galois cover
with group $\Gamma$ such that $\upX$ is quasi-projective. 
Let $X^0\subset X$ be a dense, Zariski open subset and $g^0:X^0\to Z^0$  a 
flat, proper morphism
with connected general fiber $F$ such that $\pi$
induces a finite covering $\tilde F\to F$.  Let  
$\tilde g^0: \tilde X^0\to \tilde Z^0$
be the corresponding flat, proper morphism with general fiber $\tilde F$.
Then (at least) one of the following holds:
\begin{enumerate}
\item $\tilde g^0$  extends to a locally trivial, $\Gamma$-equivariant fibration
$\tilde g:\tilde X\to \tilde Z$, or
\item $\tilde X$ contains a closed $\Gamma$-invariant subvariety
that is disjoint from a general fiber of  $\tilde g^0$.
\end{enumerate}
\end{lemma}

\begin{proof}
Let $L$ be a line bundle that is the restriction of an ample line bundle on 
some  completion $\upX \subset \barX$ to $\upX$. 

By assumption  
 $\pi$
induces a finite covering $\tilde F_v\to F_v$,
say of degree $m$, on the fibers of $g$.
Let $U\to V$ be  a flat, proper morphism
whose fibers are the degree $m$ \'etale covers of the fibers of 
$g$.
Let $W\subset \mor(U/V, X)$ be an irreducible component
parametrizing morphisms $U_v\to X$ whose image is a fiber of $g^0$.
Then  $\univ_W\to X$ is dominant and we can use 
Lemma \ref{min.constr.lem}.

We fix an actual $g$-fiber $\tilde F_z$. 
By Lemma \ref{min.constr.lem} applied to $M(\tilde F_z, X, d)$,
we see  that $\tilde W\cap M(\tilde F_z,\tilde  X)$ is algebraic.
Thus the image of the universal family over $\tilde W\cap 
M(\tilde F_z, X)$ gives a
constructible, $\Gamma$-invariant  subset $\loc_W(\tilde F_z,\tilde X)
\subset \tilde X$.
If $\loc_W(\tilde F_z,\tilde X)$ is not Zariski dense, then its
 closure   is a 
$\Gamma$-invariant, closed, algebraic  subset
that is disjoint from a general fiber.

Otherwise the morphism $\tilde g^0$ is 
a locally trivial fiber bundle with fiber $\Tilde F_z$
over a Zariski open subset of $Z^0$.
 Let $X^*\subset \tilde X$ be the largest open set 
over which $\tilde g^0$ extends to a locally trivial fiber bundle.
Then $X^*$ is $\Gamma$-invariant, hence if  $X^*\neq \tilde X$ then 
$\tilde X\setminus X^*$ is a $\Gamma$-invariant, closed algebraic  subset
that is disjoint from a general fiber.
Otherwise $X^*= \tilde X$ which shows (1).
\end{proof}

Since  $\tilde g:\tilde X\to \tilde Z$
is  $\Gamma$-equivariant, the $\Gamma$-action on $\tilde X$
descends to a $\Gamma$-action on $\tilde Z$.
If $\tilde F$ has no fixed point free automorphisms, then the
$\Gamma$-action on $\tilde Z$ is fixed point free,
but in general it can have finite stabilizers.
In some cases we will show that a finite index subgroup of
$\Gamma$ acts freely on  $\tilde Z$, but this does not seem to be
automatic.

\begin{lemma}\label{main.constr.bir}
Let $X$ be a projective variety and $\pi:\tilde X\to X$ 
 an infinite Galois cover
with group $\Gamma$ such that $\upX$ is quasi-projective. 
Let $X^0\subset X$ be a dense, Zariski open subset and $g^0:X^0\to Z^0$  a 
 proper, birational morphism.
Let $E^0\subset \ex(g^0)$ be a maximal dimensional irreducible component
of the exceptional set, 
 $F\subset E^0$ a general fiber of $g^0|_{E^0}$
and $E\subset X$ the closure of $E^0$.
Assume that  $\dim E+\dim F\geq \dim X$ and  $\pi$
induces a finite covering $\tilde F\to F$.

Then  $\pi^{-1}\bigl( E\bigr)$ is
an algebraic  subvariety of $\tilde X$.
\end{lemma}

\begin{proof} Let $\bar X\supset \tilde X$ be a smooth, algebraic 
 compactification.
It is clear that  $\pi^{-1}\bigl( E\bigr)$
is a closed analytic subspace of  $\tilde X$;
let $\tilde E_i$ be its irreducible components.

Assume first that each  $\tilde E_i$ is algebraic.
By (\ref{G.exist.lem}), there is a  subvariety $G\subset F$ such that
the intersection number $(G\cdot E)_X\neq 0$.
Set $G_i:=\tilde E_i\cap \pi^{-1}(G)$. Then
$\bigl(G_i\cdot \tilde E_i\bigr)_{\bar X}\neq 0$ for every $i$ and
$\bigl(G_i\cdot \tilde E_j\bigr)_{\bar X}\neq 0$ for  $i\neq j$.
Thus the homology classes of the closures
$[\bar E_i]\in H_*\bigl(\bar X, \q\bigr)$ are
linearly independent, and therefore $ \pi^{-1}\bigl( E\bigr)$ has only
finitely many  irreducible components. Each is algebraic by assumption, thus
 $\pi^{-1}\bigl( E\bigr)$ is
an algebraic  subvariety of $\tilde X$. 

Thus it remains to show that  each  $\tilde E_i$ is algebraic.

Let $E^1\subset E^0$ be an open subset
such that $f^1:=f|_{E^1}:E^1\to f(E^1)$ is proper and flat.
By assumption  
 $\pi$
induces a finite covering $\tilde F_v\to F_v$
say of degree $m$, on each connected component of a fiber of $f^1$.
Let $U\to V$ be  a flat, proper morphism
whose fibers are the degree $m$ \'etale covers of these $F_v$.

Note that if a morphism
$\phi:\tilde F_v \to X$ maps to a fiber of $f^1$
then so does every small deformation of it.
Thus, for each $i$, there is an irreducible component
$W_i\subset \mor(U/V, \tilde X)$ such that
$\Phi\bigl(\univ_{W_i}\bigr)\subset \tilde X$
is a Zariski dense constructible subset of $E_i$.
Therefore every  $E_i$ is algebraic.
\end{proof}

\begin{lemma}\label{G.exist.lem}
 Let $f:X\to Y$ be a projective, birational morphism, $X$ smooth.
Let $E\subset \ex(f)$ be a maximal dimensional irreducible component
of the exceptional set 
and $F\subset E$ a general fiber of $f|_{E}$.
The following are equivalent.
\begin{enumerate}
\item There is a subvariety $G\subset F$ such that
the intersection number $(G\cdot E)_X\neq 0$.
\item   $\dim E+\dim F\geq \dim X$.
\end{enumerate}
\end{lemma}

\begin{proof} Note first that although $X$ and $E$ are  not assumed compact,
the intersection number $(G\cdot E)_X$ is defined
where the subscript indicates that we compute the
 intersection number in $X$.
If it is nonzero then $\dim E+\dim G = \dim X$, thus (1)
implies (2).

To see the converse, note that 
if we take a general hyperplane section of $Y$ 
and replace $X$ by its preimage, 
the inequality in (2) remains valid. Thus, 
after taking $\dim f(E)$ hyperplane sections, we can suppose that $E=F$ maps
to a point and $2 \dim F \geq \dim X$. Next we
take hyperplane sections of $X$.
After $r$ steps, eventually  we are reduced to consider 
$f_r:X_r\to Y_r$ such  that the exceptional set $E_r$ maps
to a point and $2 \dim E_r = \dim X_r$.
Set $G=E_r$. Then
$\bigl(G\cdot E\bigr)_X=\bigl(G\cdot G\bigr)_{X_r}$ and  
$(-1)^{\dim G} \bigl(G\cdot G\bigr)_{X_r}>0$
 by \cite[Thm.2.4.1]{CM02}.
\end{proof}

\section{Proofs of the main results}

\begin{proof}[Proof of Proposition \ref{fg.ab.prop}]
Let $X$ be a normal, projective variety and $\tilde X\to X$
a quasi-projective Galois cover with group $\Gamma$.
We study where $X$ fits into the birational classification plan
of varieties and we show that many cases 
would lead to a lower dimensional 
normal, projective variety $Y$ and 
a quasi-projective Galois cover $\tilde Y\to Y$ with group $\Gamma_Y$
that is a finite index subgroup of $\Gamma$.

After several such tries, we see that there is no place for the smallest
dimensional example, unless $\Gamma$ is almost Abelian.

{\it Step 1: $X$ is smooth.}
First we claim that $\upX$ has no nontrivial, 
closed,  subvariety  invariant under a
 finite index subgroup  $\Gamma'\subset \Gamma$. Indeed, given such $\tilde W$
with irreducible components  $\tilde W_i$, each of them
is invariant  under a
 finite index subgroup  $\Gamma_i\subset \Gamma$.
Taking the  normalization $\tilde W^n_i$, 
we would get a smaller dimensional example
$W^n_i:=\tilde W^n_i/\Gamma_i$ as in (\ref{fg.ab.prop}); a contradiction.

Since $\sing \upX\subset \upX$ is algebraic and $ \Gamma$-invariant,
we conclude that $X$ is smooth.

{\it Step 2: $K_X$ is nef.}
If $K_X$ is not nef, there is an extremal contraction 
$g:X\to Z$ \cite[Thm.3.7]{KM98}.

Assume first that $g$ is not birational and let $F\subset X$ denote a smooth
fiber. Then $F$ is a smooth Fano variety, 
in particular it is rationally connected \cite{Ca92, KMM92}.
Rationally connected manifolds are simply connected, 
so the fibre $F$ lifts to $F\into \tilde X$.  
Thus, by (\ref{main.constr}),
there is a locally trivial fiber bundle
$\tilde X\to \tilde Z$ with fiber $F$.
The variety  $\tilde Z$ is quasi-projective by (\ref{qp.quot.lem}).
Note that $F$ does not admit fixed point free actions by any finite group: 
the \'etale quotient
would also be rationally connected, so simply connected.
Therefore the stabilizer $\stab_{\Gamma}(F_z)$ is trivial for
every fiber $F_z$ of  $\tilde X\to \tilde Z$.
Hence the $\Gamma$-action 
 descends to a free $\Gamma$-action on $\tilde Z$;
 a contradiction to the minimality
 of the dimension of $X$.

Assume next that  $g$ is birational. 
Let $E\subset \ex(g)$ be a maximal dimensional irreducible component
of the exceptional set 
and $F\subset E$ a general fiber of $g|_{E}$.
By the Ionescu-Wi\'sniewski inequality (see for instance \cite[Thm.2.3]{AW}) 
one has $\dim F+\dim E\geq \dim X$.
By \cite[7.5]{Kol93} and \cite[Thm.1.2]{Tak03}, the map
 $\pi_1(X)\to \pi_1(Z)$ is an isomorphism, thus
the embedding $F\into X$ lifts to
$F\into \tilde X$. 

Therefore (\ref{main.constr.bir})
implies that $\pi^{-1}(E)\subset \tilde X$ is an 
 algebraic and $ \Gamma$-invariant subset of $\tilde X$.
This is again a contradiction, thus $K_X$ is nef.

{\it Step 3: The Iitaka fibration.}
If  $K_X$ is not semiample, then we are done.
Otherwise  $K_X$ is semiample and for sufficiently divisible $m>0$,
the sections of $\o_X(mK_X)$ define  a morphism
(called the Iitaka fibration) 
$\tau:X\to I(X)$ with connected fibers such that
$\o_X(mK_X)\cong \tau^*M$ for some ample line bundle $M$.
By the adjunction formula,  $mK_F\sim 0$ for any smooth fiber
$F\subset X$.

Ideally we would like to apply (\ref{main.constr}) to 
$\tau:X\to I(X)$ and conclude that $X\cong I(X)$.
However, in general $\tilde X\to X$ induces an
infinite cover of $\tilde F\to F$, hence (\ref{main.constr}) does not apply.
Instead  we first study  compact subspaces of $\tilde F$
and then move on to the case when  $\tilde F$
has no positive dimensional  compact subspaces.


{\it Step 4: Excluding compact subspaces of $\tilde F$.}
Here we  prove that there is a finite, \'etale cover
$ F'\to F$ that is an Abelian variety and
there is no positive dimensional subvariety
$ B\subset  F$ such that 
$\pi_1\bigl( B\bigr)\to \Gamma$ has finite image.

This could be done in one step, but it may be more transparent
to handle the two assertions separately.

As a consequence of  the Beauville-Bogomolov decomposition theorem,
a suitable finite, \'etale, Galois cover $ F'\to F$
with group $G$
admits a  $G$-equivariant 
morphism $ F'\to A$ where $A$ is an Abelian variety
and the fibers are simply connected.
Thus $F\to A/G$ 
 is a  morphism whose general fibers
have  torsion canonical class and finite fundamental group.
Moreover, at least over a dense open subset of $X^0\subset X$, these
maps give a proper, smooth morphism
$X^0\to Z^0$ whose fibers $D$   have  torsion canonical class and 
 finite fundamental group.
Thus, by (\ref{main.constr}), there is a 
locally trivial fiber bundle
$\tilde\tau: \tilde X\to \tilde Z$ with fiber $\tilde D$ 
for some finite \'etale cover $\tilde D\to D$.




By Lemmas \ref{bundles.lem} and \ref{aut.gp.lem},
 by passing to a finite cover of $\tilde X$
we can assume that $\tilde X\cong \tilde D\times  \tilde Z$
and the product decomposition is unique. Thus the 
$\Gamma$-action on $\tilde X$ gives a homomorphism
of $\Gamma$ to $\aut\bigl(\tilde D\bigr)$ with finite image.
So there is a finite index subgroup
$\Gamma_0\subset \Gamma$ that acts trivially on $\tilde D$.
Then for any $p\in  \tilde D$, the section
$\{p\}\times  \tilde Z$ is quasi-projective and 
$\Gamma_0$-invariant; again a contradiction.

Thus now we know that a general fiber $F$ of
$\tau: X\to I(X)$ has a finite, \'etale cover $F'\to F$ 
that is an Abelian variety. 

Assume next that for every general fiber $ F$ 
there are    positive dimensional subvarieties
$B_i\subset  F'$ such that 
$\pi_1\bigl( B_i\bigr)\to \Gamma$ has finite image.
The largest dimensional such subvarieties are an Abelian subvariety
$B'\subset F'$ and its translates.

Consider the relative $\Gamma$-Shafarevich map for  $X\to I(X)$ 
\cite[3.10]{Kol93}.
We get a dense open set $X^0$ and a smooth, proper morphism
$\rho:X^0\to Y^0$ such that $\pi_1\bigl( B\bigr)\to \Gamma$ has finite image
for every fiber $B$ of $\rho$.
(Moreover, $X\to I(X)$ factors through $\rho$ and 
$\rho$ is universal with these properties). 
The preimage of $B$ in $F'$ is a translate of $B'$.

As before,   we can apply  (\ref{main.constr}) to $X^0\to Y^0$.
Thus we obtain a locally trivial fiber bundle
$\tilde X\to \tilde Y$ with fiber $\tilde B$.
By Lemma \ref{bundles.lem} we may assume that $\tilde X\to \tilde Y$
is topologically trivial. In particular,
$\pi_1(\tilde X)=\pi_1(\tilde Y)+\pi_1(\tilde B)$.

Thus, by passing to a finite Galois cover of $\tilde B$,
we can assume that $\pi_1\bigl( \tilde B\bigr)\to \Gamma$
is the constant map, hence $\Gamma$ is a quotient
of $\pi_1(\tilde Y)$. Thus the free  $\Gamma$-action on $\tilde X$
descends to a free  $\Gamma$-action on $\tilde Y$.
This again contradicts the assumption on the minimality of $\dim X$.



Thus we conclude that the general fiber  $F$ of  the Iitaka fibration
$\tau:X\to I(X)$ has a finite \'etale cover 
$ F'\to  F$ that is an Abelian variety
and a very general fiber $F$ 
has no positive dimensional subvariety
$ B\subset  F$ such that 
$\pi_1\bigl( B\bigr)\to \pi_1(X)\to \Gamma$ has finite image.
Thus, in the terminology of \cite{Kol93}, $X$ has
generically large fundamental group on $F$.

{\it Step 5: Abelian schemes.} 
This part of  the proof closely follows \cite{Nak99}.
By   \cite[5.9 and 6.3]{Kol93}, $X$ has a finite \'etale cover
$X_1\to X$ that is birational to a smooth projective variety
$X_2$ such that the  Iitaka fibration
$\tau_2:X_2\to I(X_2)$ is smooth with Abelian fibers and general type base.
We are thus in position to apply the Kobayashi-Ochiai theorem (\ref{KO}): the image 
$\mathrm{Im}\bigl(\pi_1(F)\to \Gamma\bigr)$ has finite index in $\Gamma$. We obtain the final contradiction since this implies that
$\Gamma$  is almost abelian.
 \end{proof}

\begin{proof}[Proof of Theorem  \ref{alb.general.thm}] 
Consider the Stein factorization
$X\to B\to A$.

If there is a map $B \rightarrow Y$ such that $Y$ is of general type,
we know by (\ref{lemmaueno}) that a finite \'etale cover is a product of a variety of general 
type $Y'$ and an Abelian variety.
By  (\ref{KO}) the group $\Gamma$ induces a finite covering on $Y'$, so
$B \times_A \tilde A$ has a finite cover that
is a product of a variety of general type and a cover of an Abelian variety.
In particular $B \times_A \tilde A$, hence $\tilde A$, contains compact analytic subvarieties.
It follows by (\ref{lemmaueno}) that $B$ is an Abelian variety. 

Assume that $g:B\to A$ is not surjective.
Then  $g(B)\subsetneq  A$ is an  Abelian subvariety
and by Poincar\'e's theorem there is an Abelian subvariety 
$C\subset A$ such that
$C\cap g(B)$ is finite. Moreover, 
$\pi_1(B)\to \pi_1(A)$ has infinite index image but
$\pi_1(C)+\pi_1(B)\to \pi_1(A)$ has finite index image.
By assumption $\tilde A\to A$ induces an infinite degree
cover of $C$, thus $\pi^{-1}(B)$ has infinitely many connected components.
This is impossible since $\tilde X$ is quasi-projective.
Thus $g:B\to A$ is  surjective.
Therefore we can  replace $A$ with $B$ and assume to
start with that $\alpha:X\to A$ is surjective with
connected fibers.  

Consider first the case when $\alpha$ is birational.
If $X$ is singular then consider
$\bigl(\sing X)^n\to A$, the normalization of the singular locus mapping to
$A$.  By induction on the dimension,   (\ref{alb.general.thm}) 
applies. This map is, however, not even surjective. Thus $X$ is smooth. 
Let  $E\subset X$ denote an irreducible component of the
 exceptional divisor $\ex(\alpha)$ such that
$\dim \alpha(E)$ has maximal possible dimension.
 By (\ref{main.constr.bir}),
its preimage $\tilde E\subset  \tilde X$ is quasi-projective.
By induction on the dimension we get that
  $\alpha|_E:E\to A$ is a
locally trivial fiber bundle. But this map is not even surjective.
Thus $\alpha$ has no exceptional divisors and therefore it is an isomorphism.

If  $\alpha$ is not birational, let $F\subset X$ be a general fiber.
Suppose that there exists a closed $\Gamma$-invariant subvariety $Z \subset \tilde X$ that is disjoint from $F$.
Let $Z^n$ be the normalisation,  then $Z^n/\Gamma$  is a normal, projective variety with a morphism
to the abelian variety $A$. Since $Z$ is quasi-projective we see again by induction on the dimension
that $Z^n/\Gamma \rightarrow A$ is surjective. Thus $Z$ meets $F$, a contradiction.

Thus we know by (\ref{main.constr}) that $\tilde X$  is a
locally trivial fiber bundle $\tilde \tau:\tilde X\to \tilde Z$ with fiber $F$.
By  (\ref{qp.quot.lem}), $\tilde Z$ is quasi-projective and 
  $\tilde \tau$ factors through $\tilde\alpha$. By taking the
quotient we obtain
$$
\alpha: X\stackrel{\tau}{\to} Z\stackrel{\alpha_Z}{\to} A.
$$
Note that by construction $Z=\tilde Z/\Gamma$ is a  normal complex space
and $\tau:X\to Z$ is proper and equidimensional.
Thus by (\ref{qp.quot.lem}), $Z$ is a projective variety. 
We already saw that these imply that  $\alpha_Z$ is an isomorphism.
Thus $\tilde Z=\tilde A$, $\tilde \alpha: \tilde X\to \tilde A$
is a locally trivial fiber bundle
and so is
$\alpha:X\to A$.
\end{proof}

\begin{proof}[Proof of Corollary \ref{better.main.thm}] 
Let $\Gamma$ be the Galois group of the cover $\tilde X\to X$.
By (\ref{fg.ab.prop}) the group $\Gamma$ is almost abelian, so there exists 
an intermediate finite \'etale, Galois  cover $X'\to X$
such that $\tilde X \rightarrow X'$ is Galois with a Galois group $\Gamma'$ that 
is free abelian.  Let $A$ be the Albanese torus of $X'$.
By (\ref{alb=ab.lem}) the group $\Gamma'$
is a quotient of $H_1(A, \Z)$. 
The maximal sub-Hodge structure contained in the kernel of
 $H_1(A, \Z) \rightarrow \Gamma'$ corresponds to an Abelian subvariety $B \subset A$.
Thus up to replacing $A$ by $A/B$
we can suppose that this sub-Hodge structure is zero.
Hence if $\pi: \tilde A \rightarrow A$ is the Galois cover corresponding to the group $\Gamma'$, 
the quasi-projective variety $\tilde A$ does not have compact analytic subvarieties.
Conclude with (\ref{alb.general.thm}).
\end{proof}

Next we give examples 
of fiber bundles $X\to E$ over elliptic curves whose
universal cover $\tilde X\to \tilde E=\C$ is quasi-projective,
yet the deck transformations can not 
 be chosen algebraic. 
 
\begin{example}\label{nonalg.action.exmp}
We start with a noncompact example.

Let $Y$ be a $\c^*$-bundle over an elliptic curve
$E$. Pulling it back to $\tilde E\cong \c$ we get
the trivial bundle $\c\times \c^*\to \c$.

Note that every algebraic morphism $\c\to \c^*$ is constant,
hence  every algebraic
automorphism of
$\c\times \c^*$ that commutes with the first projection
is of the form  $(x,y)\mapsto (x,cy)$.
In particular, they preserve the flat structure.
 Thus if $Y\to E$ is a quotient
of $\c\times \c^*$ by algebraic deck transformations
then the flat structure on $\c\times \c^*\to \c$
descends to a flat  structure on $Y\to E$.
In particular, $c_1(Y)=0$ and $Y\to E$ is topologically trivial.

By contrast, the biholomorphisms
$(x,y)\mapsto (x,e^{g(x)}y)$ of $\c\times \c^*$
commute with the first projection, and
every $\c^*$-bundle over
$E$ is the quotient of  $\c\times \c^*$
by holomorphic  deck transformations.

To get compact examples out of these, let $F$ be a projective variety
such that the connected component of $\aut(F)$ is
$\bigl(\c^*\bigr)^m$ for some $m>0$.
(For instance, $F$ can be the blow up of $\p^2$ at 3 non-collinear points.)
By the above arguments, if
$X\to E$ is a locally trivial $F$-bundle that is a
 quotient
of $\c\times F$ by algebraic deck transformations,
then $X\to E$ is topologically trivial after a finite cover $E'\to E$.
(In fact, $X\to E$ is itself 
topologically trivial if $\aut(F)=\bigl(\c^*\bigr)^m$.)
On the other hand,  $F$-bundles
obtained from a line bundle with nonzero Chern class
 do not have this property.
\end{example}

\section{Auxiliary results}

Here we collect various theorems that were used
during the proofs. The most important one 
is a consequence of
the Kobayashi-Ochiai theorem 
which asserts that a  meromorphic map from a 
quasi-projective variety
to a  variety of general type can not have essential singularities.

\begin{theorem} \cite[Thm.2]{KO75} \label{KO-orig}
Let $Y$ be a projective variety of general type, $V$  a complex manifold, and
 $B \subset V$  a proper closed analytic subset. Let 
\merom{f}{V \setminus B}{Y} be a nondegenerate meromorphic map.
(That is,  such that the tangent map
$T_{V \setminus B} \rightarrow T_Y$  is surjective at at least one point 
$v \in V \setminus B$.)
 Then $f$ extends to a meromorphic map $V \dashrightarrow Y$.
\end{theorem}

Since a fiber of a  meromorphic map $V \dashrightarrow Y$
has only finitely many irreducible components, this
immediately implies the following.

\begin{corollary}  \label{KO}
Let $X$ be a quasi-projective variety and $f:X\map Y$
 nondegenerate meromorphic map
from $X$ to a variety of general type.
Let $F\subset X$ be an irreducible component of any fiber of $f$.

Let 
$\tilde X\to X$  be an \'etale Galois cover with group $\Gamma$.
 If $\tilde X$ is 
Zariski open in a compact complex manifold
then  $\im[\pi_1(F)\to \pi_1(X)\to\Gamma ]$ has finite index in $\Gamma$.  \qed
\end{corollary}

In a special case, the above conclusion can be strengthened
much further.

\begin{theorem} \label{lemmaueno} 
\cite[Thm.13]{Kaw81}
Let $A$ be an Abelian variety and $X \to A$ a finite morphism from a
normal projective variety to $A$.
If $X$ does not map onto a variety of general type, then $X$ is 
an Abelian variety.  Otherwise a finite \'etale cover of $X$ is a product of
a variety of general type and an Abelian variety. \qed
\end{theorem}

The splitting mentioned in the statement of 
Theorem \ref{theoremquasiprojective} is a straightforward consequence of
 deep results of Grauert.

\begin{theorem} \label{theoremisotrivial}
Let \holom{f}{X}{Y} be a locally trivial 
proper fibration between complex manifolds. If the universal cover of $Y$ is Stein and contractible, then the universal cover of $X$ splits as a product:
$$\upX\simeq \tilde{F}\times\upY.$$
\end{theorem}

\begin{proof}
Since $f$ is locally trivial and proper, 
it is a fiber bundle with fiber $F$ and group $G=\mathrm{Aut}(F)$ (a complex Lie group). Consider the fiber product
$$\upX_f=X\times_Y \upY;$$
it is a connected cover of $X$ which is also a fiber bundle over $\upY$ (fiber $F$ and group $G$). We can now apply \cite[Satz 6]{Gra58}: this fiber bundle has to be trivial and this gives a splitting
$$\upX_f\simeq F\times\upY.$$
Since $\upX_f$ is an intermediate cover, $\upX$ has to split as well.
\end{proof}

If the base of a fiber bundle is not known to be contractible,
one can still prove global topological triviality of certain fiber bundles.

\begin{lemma} \label{bundles.lem}
Let $V$ be a complex manifold and 
$\pi: U\to V$  a complex analytic fiber bundle
with compact fiber $F$. Assume that the structure group $G$ is
compact  (hence its connected component $G^{\circ}$ is a complex torus).
Assume furthermore that there is a closed subspace
$W\subset U$ such that $\pi:W\to V$ is generically finite.

Then there is a finite \'etale cover $\sigma: V'\to V$ such that
the pull-back 
$\sigma^*U\to  V'$ is globally trivial as a $C^{\infty}$-fiber bundle.

If $G$ is finite then $\sigma^*U\to  V'$ is  complex analytically trivial. 
\end{lemma}

\begin{proof}  We have a monodromy representation
$\pi_1(V)\to G/G^{\circ}$. By passing to the cover of $V$ corresponding to
its kernel, we may assume that  the structure group $G$ is a complex torus.
If $\dim G=0$ then we have a trivial bundle.

In general,  $G$ is diffeomorphic to $\bigl(S^1\bigr)^{2d}$ , thus  
$C^{\infty}$-fiber bundles
with  structure group $G$ are classified by
$$
c_1(U/V)\in H^1\bigl(V, G\bigr)=H^1\bigl(V, S^1\bigr)^{2d}=
H^2\bigl(V, \z\bigr)^{2d}.
$$
Let $Z\subset V$ denote the closed subspace over which $\pi:W\to V$
has positive dimensional fibers. Then $Z$ has complex codimension $\geq 2$,
thus $H^2\bigl(V, \z\bigr)=H^2\bigl(V\setminus Z, \z\bigr)$.
Therefore we can replace $V$ by $V\setminus Z$ and assume that
 $\pi:W\to V$ is  finite.

After base change to $W$, the fiber bundle has a section, thus its
Chern class is trivial. This is equivalent to 
$$
c_1(U/V)\in \ker\Bigl[H^2\bigl(V, \z\bigr)
\stackrel{\pi^*}{\to} H^2\bigl(W, \z\bigr)\bigr]^{2d}.
$$
With $\q$-coefficients, the map $\pi^*$ is an injection,
thus $c_1(U/V)$ is torsion in $H^2\bigl(V, \z\bigr)^{2d}$.
The torsion in $H^2\bigl(V, \z\bigr)$ comes from the torsion
in $H_1\bigl(V, \z\bigr)$, hence it is killed after a suitable
finite \'etale cover of $V$.
\end{proof}

Examples where the assumptions of (\ref{bundles.lem})
hold are given by the following.
(See \cite[\S 14]{ueno} for a more modern exposition.)

\begin{lemma} \cite{matsusaka, matsumura} \label{aut.gp.lem}
Let $X$ be a 
normal, projective variety that is not birationally ruled.
Let $L$ be an ample line bundle and let $\aut(X,L)$ denote the group
of those automorphisms $\tau:X\to X$ such that $\tau^*L$ is
numerically equivalent to $L$.
Then the identity component $\aut^{\circ}(X,L)\subset \aut(X,L)$
is an Abelian variety and the quotient $ \aut(X,L)/\aut^{\circ}(X,L)$
is finite. \qed
\end{lemma}

Recall that for a normal projective variety $X$, the Albanese morphism is defined as
the universal map from $X$ to abelian varieties.
The following result is a straightforward consequence of the analytic construction 
of the Albanese morphism for smooth projective varieties but still holds for normal ones. 

\begin{lemma}\label{alb=ab.lem} Let $X$ be a normal, projective variety
and $\alb:X\to \Alb(X)$ the Albanese morphism.
Then $\alb_*:H_1(X,\z)\to H_1(\Alb(X),\z)$
is surjective with finite kernel.
\end{lemma}

\begin{proof}

Take a resolution of singularities
$g:Y\to X$.  Set   $A_Y:=\Alb(Y)$.

Let $\tilde X\to X$ be the Galois cover
corresponding to $H_1(X,\z)/(\mbox{torsion})$. 
It induces a  Galois cover  $\tilde Y\to Y$ whose
Galois group is again $H_1(X,\z)/(\mbox{torsion})$.
Thus there is a  Galois cover  $\tilde A_Y\to A_Y$
such that $\tilde Y=\tilde A_Y\times_{A_Y}Y$.

Let $F_x\subset Y$ be any fiber of $g$.
By construction, $\tilde Y\to Y$ is trivial on $F_x$, hence
$\tilde A_Y\to A_Y$ is trivial on $\alb_Y(F_x)$. 

Let $B_Y\subset A_Y$ be the smallest Abelian subvariety
such that every $\alb_Y(F_x)$ is contained in a translate
of $B_Y$. Then $\tilde A_Y\to A_Y$ is trivial on $B_Y$,
hence $\tilde A_Y$ is a pull back of the corresponding  cover
$\tilde A_X\to  A_X:=A_Y/B_Y$.

By construction, every fiber of $g$ maps to a point in
$A_X$, thus $Y\to A_Y$ descends to a morphism
$X\to A_X$ and $\tilde X=\tilde A_X\times_{A_X}X$.
Thus $A_X=\Alb(X)$ and we are done.
\end{proof}

\begin{lemma}\label{qp.quot.lem} Let $f:X\to Z$ be a proper, equidimensional
morphism of normal complex spaces. Assume that  $X$ is quasi-projective.
Then $Z$ has a unique  quasi-projective structure such that $f$
is an algebraic morphism.
\end{lemma}

\begin{proof} The map $f$ determines a natural morphism $Z\to \Chow{X}$
which maps $Z$ biholomorphically onto a connected component of
$\Chow{X}$. We  thus need to identify $Z$ with its image.
\end{proof}


\begin{thebibliography}{EKPR09}

\bibitem[AK01]{AK01}
Yukitaka Abe and Klaus Kopfermann.
\newblock {\em Toroidal groups}, volume 1759 of {\em Lecture Notes in
  Mathematics}.
\newblock Springer-Verlag, Berlin, 2001.
\newblock Line bundles, cohomology and quasi-abelian varieties.

\bibitem[Akh95]{Ak95}
Dmitri~N. Akhiezer.
\newblock {\em Lie group actions in complex analysis}.
\newblock Aspects of Mathematics, E27. Friedr. Vieweg \& Sohn, Braunschweig,
  1995. 


\bibitem[AW97]{AW}
Marco Andreatta and Jaros{\l}aw~A. Wi{\'s}niewski.
\newblock A view on contractions of higher-dimensional varieties.
\newblock In {\em Algebraic geometry---{S}anta {C}ruz 1995}, volume~62 of {\em
  Proc. Sympos. Pure Math.}, pages 153--183. Amer. Math. Soc., Providence, RI,
  1997.

\bibitem[Bea83]{Beau83}
Arnaud Beauville.
\newblock Vari\'et\'es {K}\"ahleriennes dont la premi\`ere classe de {C}hern
  est nulle.
\newblock {\em J. Differential Geom.}, 18(4):755--782 (1984), 1983.

\bibitem[BK98]{BK98}
F.~Bogomolov and L.~Katzarkov.
\newblock Complex projective surfaces and infinite groups.
\newblock {\em Geom. Funct. Anal.}, 8(2):243--272, 1998.

\bibitem[CC91]{CC91}
F.~Capocasa and F.~Catanese.
\newblock Periodic meromorphic functions.
\newblock {\em Acta Math.}, 166(1-2):27--68, 1991.

\bibitem[Cam92]{Ca92}
F.~Campana.
\newblock Connexit\'e rationnelle des vari\'et\'es de {F}ano.
\newblock {\em Ann. Sci. \'Ecole Norm. Sup. (4)}, 25(5):539--545, 1992.

\bibitem[dCM02]{CM02}
Mark Andrea~A. de~Cataldo and Luca Migliorini.
\newblock The hard {L}efschetz theorem and the topology of semismall maps.
\newblock {\em Ann. Sci. \'Ecole Norm. Sup. (4)}, 35(5):759--772, 2002.

\bibitem[EKPR09]{EKPR}
Philippe Eyssidieux, Ludmil Katzarkov, Tony Pantev, and Mohan Ramachandran.
\newblock Linear {S}hafarevich {C}onjecture.
\newblock Preprint arXiv0904.0693, 2009.

\bibitem[Gra58]{Gra58}
Hans Grauert.
\newblock Analytische {F}aserungen \"uber holomorph-vollst\"andigen {R}\"aumen.
\newblock {\em Math. Ann.}, 135:263--273, 1958.

\bibitem[Kaw81]{Kaw81}
Yujiro Kawamata.
\newblock Characterization of abelian varieties.
\newblock {\em Compositio Math.}, 43(2):253--276, 1981.

\bibitem[KO75]{KO75}
Shoshichi Kobayashi and Takushiro Ochiai.
\newblock Meromorphic mappings onto compact complex spaces of general type.
\newblock {\em Invent. Math.}, 31(1):7--16, 1975.

\bibitem[Kol92]{f&a} J{\'a}nos Koll{\'a}r (with 14 coauthors)
{\em Flips and abundance for algebraic threefolds}.
\newblock Salt Lake City, Utah, August 1991, Ast{\'e}risque No.
  211, \newblock Soci\'et\'e Math\'ematique de France, Paris, 1992.

  
  \bibitem[Kol93]{Kol93}
J{\'a}nos Koll{\'a}r.
\newblock Shafarevich maps and plurigenera of algebraic varieties.
\newblock {\em Invent. Math.}, 113(1):177--215, 1993.

\bibitem[Kol95]{kol-book}
J{\'a}nos Koll{\'a}r.
\newblock {\em Shafarevich maps and automorphic forms}.
\newblock M. B. Porter Lectures. Princeton University Press, Princeton, NJ,
  1995.

\bibitem[KM98]{KM98}
J{\'a}nos Koll{\'a}r and Shigefumi Mori.
\newblock {\em Birational geometry of algebraic varieties}, volume 134 of {\em
  Cambridge Tracts in Mathematics}.
\newblock Cambridge University Press, Cambridge, 1998.
\newblock With the collaboration of C. H. Clemens and A. Corti.

\bibitem[KMM92]{KMM92}
J\'anos Koll\'ar, Yoichi Miyaoka, and Shigefumi Mori.
\newblock Rational connectedness and boundedness of {F}ano manifolds.
\newblock {\em J. Diff. Geom. 36}, pages 765--769, 1992.

\bibitem[Mat63]{matsumura}
 Hideyuki Matsumura.
\newblock On algebraic groups of birational transformations.
\newblock {\em Atti Accad. Naz. Lincei Rend. Cl. Sci. Fis. Mat. Natur.  34},
 pages  151--155, 1963.

\bibitem[Mat58]{matsusaka}
Teruhisa Matsusaka.
\newblock  Polarized varieties, fields of moduli and generalized Kummer varieties of polarized abelian varieties.
\newblock {\em Amer. J. Math. 80}, pages 45--82,   1958.

\bibitem[Nak99]{Nak99}
Noboru Nakayama.
\newblock Projective algebraic varieties whose universal covering spaces are
  biholomorphic to {${\bf C}^n$}.
\newblock {\em J. Math. Soc. Japan}, 51(3):643--654, 1999.

\bibitem[Rei87]{reid-minmod}
Miles Reid.
\newblock Tendencious survey of {$3$}-folds.
\newblock In {\em Algebraic geometry, {B}owdoin, 1985}, volume~46 of {\em Proc. Sympos. Pure Math.}, pages 333--344. Amer.
  Math. Soc., Providence, RI, 1987.

\bibitem[Sha74]{Sha74}
I.~R. Shafarevich.
\newblock {\em Basic algebraic geometry}.
\newblock Springer-Verlag, New York, 1974.
\newblock  Die Grundlehren der
  mathematischen Wissenschaften, Band 213.

\bibitem[Tak03]{Tak03}
Shigeharu Takayama.
\newblock Local simple connectedness of resolutions of log-terminal
  singularities.
\newblock {\em Internat. J. Math.}, 14(8):825--836, 2003.

\bibitem[Uen75]{ueno}
 Kenji Ueno.
\newblock Classification theory of algebraic varieties and compact complex spaces. Lecture Notes in Mathematics, Vol. 439. 
\newblock  Springer-Verlag, Berlin-New York, 1975. xix+278 pp. 

\bibitem[Win98]{Win98}
J{\"o}rg Winkelmann.
\newblock Complex analytic geometry of complex parallelizable manifolds.
\newblock {\em M\'em. Soc. Math. Fr.}, 1998.

\end{thebibliography}

\end{document}